\documentclass[11pt]{article}
\usepackage{amsmath, amssymb, amsthm}
\usepackage{enumerate}
\usepackage{tikz-cd}
\usepackage{cancel}
\usepackage{empheq}

\newtheorem{theorem}{Theorem}[section]
\newtheorem{proposition}[theorem]{Proposition}
\newtheorem{lemma}[theorem]{Lemma}
\newtheorem{definition}[theorem]{Definition}
\theoremstyle{definition}
\newtheorem{example}[theorem]{Example}
\theoremstyle{remark}
\newtheorem{remark}[theorem]{Remark}

\newcommand{\Z}{\mathbb{Z}}

\newcommand{\Q}{\mathbb{Q}}
\newcommand{\C}{\mathbb{C}}
\newcommand{\PP}{\mathbb{P}}
\newcommand{\F}{\mathbb{F}}
\newcommand{\spec}{\operatorname{Spec}}
\newcommand{\Prim}{\mathrm{Prim}}
\newcommand{\dR}{\mathrm{dR}}
\newcommand{\rig}{\mathrm{rig}}

\title{Picard-Fuchs equations of the generalized Dwork family}
\author{Ryo Negishi\footnote{Department of Mathematics, Graduate School of Science, Hokkaido University, Sapporo 060-0810, JAPAN E-mail: negishi.ryo.r2@elms.hokudai.ac.jp}}
\date{}

\numberwithin{equation}{section}

\begin{document}

\maketitle

  \begin{abstract}
  We determine the Picard-Fuchs equations of the generalized Dwork families by Katz.
As an application, we compute the Frobenius matrix on the rigid cohomology of the family.
This was originally done by Kloosterman, while
we give an alternative computation with use of the Picard-Fuchs equations.
\end{abstract}

\section{Introduction}
The Dwork family is the one parameter family of hypersurfaces of the projective space $\PP^{n-1}$ given by a homogeneous equation
\begin{equation}
  \label{eq:dwork}
  X_1^n+ \cdots + X_n^n - n \lambda X_1 \cdots X_n = 0
\end{equation}
where $\lambda$ is the parameter. The name ``Dwork family''  comes from 
the fact that Bernard Dwork preferred to use this family in his study
on the Frobenius actions of algebraic families over a finite field
(cf. \cite{katz}).

In his paper \cite{katz}, Katz introduces a generalization 
of the Dwork family
  \[\pi : \mathbb{X} \longrightarrow \mathbb{A}^1 - \mu_d\]
defined by a homogeneous equation
 \begin{equation}\label{intro-eq1}
   w_1X_1^d + \cdots + w_nX_n^d- d \lambda X_1^{w_1} \cdots X_n^{w_n}=0,
 \end{equation}
 where $w_1, \ldots, w_n$ and $d$ are positive integers such that
 $\sum_{i=1}^n w_i =d$ and $\gcd(w_1, \ldots, w_n) = 1$.
 The group $\Gamma_W := \{ (\zeta_1, \ldots, \zeta_n) \in \mu_d \mid \zeta_1^{w_1} \cdots \zeta_n^{w_n}=1\}$ acts on the equation \eqref{intro-eq1} in a natural way 
(see \S \ref{group-action-sect}),
and this action induces the eigendecomposition of 
the primitive part of the $\ell$-adic sheaf $R^{n-2}\pi_*\overline{\Q}_\ell$. Then he proves that 
each component is isomorphic to a hypergeometric sheaf.  

Katz also discusses the Picard-Fuchs equations 
on the de Rham cohomology group.
Thanks to the result on the $\ell$-adic sheaves, 
it follows that each eigencomponent of the
de Rham cohomology group is isomorphic to a hypergeometric 
$\mathcal D$-module. However, 
it is more delicate to determine the Picard-Fuchs equations.
Katz computes the Picard-Fuchs equation of
a holomorphic form which is invariant under the action by $\Gamma_W$
(\cite[\S 8]{katz}).
This is done by computing the period of integral along a certain homology cycle.
However, it seems difficult to extend this argument for other cohomology classes.
A general method to compute the Picard-Fuchs equation of a cohomology class is
the Griffiths-Dwork method \cite[\S 5.3]{CK}.
However, this usually works well only when the defining equation is provided individually
(namely the degree and exponents of the equation are particular numbers),
and this seems not useful in our setting.
To the best knowledge of the author, the Picard-Fuchs equations of
the family \eqref{intro-eq1} have not been determined completely.

The purpose of this paper is to determine the Picard-Fuchs equations of 
the family \eqref{intro-eq1} for every components.
The main result is 
Theorem \ref{theorem:picard-fuchs-descent} 
(or equivalently Theorem \ref{theorem:picard-fuchs-reduced}).
For the proof, we use the technique of \cite[Lemma 2]{gahrs}, while we employ Katz's
results in \cite{katz} in several places.

In \S 4,
we give an application to $p$-adic cohomology.
In his paper \cite{kloosterman},
Kloosterman computes
the matrix $A(\lambda)$ describing the Frobenius action
on the de Rham cohomology of certain families including Dwork family without computing Picard-Fuchs equations
(the method used in \cite{kloosterman}
is often called the deformation method, we refer the reader to \cite{gerkmann}
for an exposition in the term of the relative rigid cohomology). 
We give an alternative proof for computing $A(\lambda)$ with use of Picard-Fuchs equations. 

\bigskip

\noindent\textbf{Acknowledgment}. The author is grateful to Professor Masanori Asakura for helpful comments and suggestions.  The author also  would like to thank Professor Seidai Yasuda for his helpful comments on the proof of Theorem \ref{theorem:picard-fuchs}. This work was supported by JST SPRING, Grant Number JPMJSP2119. 

\section{Preliminaries}
In this section, we review the generalized Dwork family from \cite{katz}.

\subsection{Definition of the generalized Dwork family}\label{defn-sect}
 Fix integers $d \geq n \geq 3$ and a $n$-tuple of positive integers $W=(w_1, \ldots, w_n)$ such that
   \begin{equation}
     \label{eq:conditions}
     \sum_{i=1}^n w_i =d,\ \gcd(w_1, \ldots, w_n) = 1.
\end{equation}
 Put $d_W = lcm(w_1, \ldots, w_n)d$. Let $\Phi_{d_W}(T)$ be the $d_W$-th cyclotomic 
 polynomial (the minimal polynomial of primitive $d_W$-th root of unity over $\Q$),
  and put \[R_0 := \Z[1/d_W][T]/(\Phi_{d_W}(T)).\]
Let
$\pi : \mathbb{X} \to \mathbb{A}^1_{R_0}:=\spec R_0[\lambda]$
be the one parameter family
  defined by the homogeneous equation
\begin{equation}
  \label{eq:generalized}
  Q_\lambda :=  w_1X_1^d + \cdots + w_nX_n^d - d \lambda X_1^{w_1} \cdots X_n^{w_n}=0.
\end{equation}
We call $\pi$ the generalized Dwork family.
\begin{lemma}[{\cite[Lemma 2.1.]{katz}}]\label{lemma:smooth}
  The morphism $\pi$ is smooth on $U:= \spec R_0[\lambda, (1-\lambda^d)^{-1}]$.
\end{lemma}

Fix a prime number $\ell$ and an embedding $R_0$ into $\overline{\Q}_\ell$. We define 
the primitive part $\Prim^{n-2}_{\overline\Q_\ell}$ on $U[1/\ell]$ as follows.
When $n-2$ is odd, set $\Prim^{n-2}_{\overline\Q_\ell}=R^{n-2}\pi_*\overline{\Q}_\ell|_{U[1/\ell]}$.
When $n-2$ is even, we set
$\Prim^{n-2}_{\overline\Q_\ell}$ to be the subsheaf of $R^{n-2}\pi_*\overline{\Q}_\ell|_{U[1/\ell]}$ which vanishes by the cup product with  $(n-2)/2$-fold product of the hyperplane class.
We have a direct decomposition
\[R^{n-2}\pi_*\overline{\Q}_\ell|_{U[1/\ell]} = 
\Prim^{n-2}_{\overline\Q_\ell}\oplus \overline{\Q}_\ell(-(n-2)/2).\]

\subsection{Group action and eigendecomposition}\label{group-action-sect}
For a commutative ring $A$, let $\mu_d(A)$ denote the group of $d$-th root of unity in $A$.
Put $\Gamma=(\mu_d(R_0))^n$ and
\[\Gamma_W := \{ (\zeta_1, \ldots, \zeta_n) \in \Gamma \mid \zeta_1^{w_1} \cdots \zeta_n^{w_n}=1\}. \]
Let $\Delta$ be the diagonal set of $\Gamma_W$. Then the finite abelian group $\Gamma_W/\Delta$
acts on the generalized Dwork family as \[(X_1, \ldots, X_n, \lambda) \mapsto (\zeta_1X_1, \ldots, \zeta_nX_n, \lambda).\]
The character group $D\Gamma = \mathrm{Hom}_{\mathrm{Group}}(\Gamma, R_0^\times)$ is identified with the group $(\Z/d\Z)^n$ via the pairing
\begin{align*}
  \Gamma &\times (\Z/d\Z)^n \longrightarrow R_0^\times\\
  (\zeta_1,\ldots, \zeta_n) &\times (v_1, \ldots, v_n) \mapsto \zeta_1^{v_1}\cdots \zeta_n^{v_n}.
\end{align*}
The character group $D(\Gamma_{W}/\Delta)$ of $\Gamma_{W}/\Delta$ is identified with the quotient of \[(\Z/d\Z)^n_0:= \{V=(v_1, \ldots, v_n) \in (\Z/d\Z)^n \mid \sum_{i=1}^nv_i = 0\}\] by the subgroup $\langle W\rangle$ of  $(\Z/d\Z)^n_0$ generated by $W$, so that the pairing
\[
\Gamma_W/\Delta\times (\Z/d\Z)^n_0/\langle W\rangle\longrightarrow R_0^\times
\]
is perfect.
For a $R_0$-module $M$ on which a finite abelian group $G=\Gamma_{W}/\Delta$ acts, 
let 
\[M_\chi = \{m \in M \mid \forall g \in G, g\cdot m = \chi(g)m\},\]
so that we have the decomposition
\[M = \bigoplus_{\chi \in DG}M_\chi.\]
Let $\Prim^{n-2}(V\bmod W)_{\overline\Q_\ell}$ denote the eigensheaf 
of the primitive part $\Prim^{n-2}_{\overline\Q_\ell}$
for $V \in (\Z/d\Z)^n_0/\langle W\rangle$. 
We have the decomposition
\[\Prim^{n-2}_{\overline\Q_\ell} = \bigoplus_{V \in (\Z/d\Z)^n_0/\langle W\rangle} \Prim^{n-2}(V\bmod W)_{\overline\Q_\ell}.\]

\subsection{Review of Katz's results}
Here we review Katz's results for the generalized Dwork family in \cite{katz}, which we will use later. 

\medskip

  We say that 
  $V=(v_1, \ldots, v_n) \in (\Z/d\Z)^n_0$ is totally nonzero if and only if
  $v_i \neq 0  \pmod d$ for all $i$.
  For each $i$, let $\tilde{v}_i\in\{0,1,\ldots,d-1\}$ be the unique integer such that
  $\tilde{v}_i\equiv v_i\bmod d$. The degree of $V$ is defiend by
  \[\deg V := \frac{1}{d}\sum^n_{i=1}\tilde{v_i}.\]

\begin{lemma}[{\cite[Lemma 3.1.]{katz}}]\label{lemma:rank}
  \[\mathrm{rank}\ \Prim^{n-2}(V\bmod W)_{\overline\Q_\ell} = \#\{r \in \Z/d\Z \mid V+rW\  \text{is totally nonzero} \}.\]
\end{lemma}

We consider the generalized Dwork family over
$U-\{0\} = \mathbb{G}_{m} - \mu_d$, and 
write it by $\pi:\mathbb{X}\to \mathbb{G}_{m, R_0} - \mu_d$ for simplicity of notation.
Let \[[d] : \mathbb{G}_{m} - \mu_d 
\longrightarrow \mathbb{G}_{m} -\{1\} = \spec R_0[t, (t(1-t))^{-1}]\] be the map corresponding to the ring morphism
\begin{align*}
  R_0[t, (t(1-t))^{-1}] & \rightarrow R_0[\lambda, (\lambda(1-\lambda^d))^{-1}]\\
t &\mapsto \lambda^{-d}.
\end{align*}
 Then there exists a descent family $\pi_{\text{\rm desc}}:\mathbb{Y} \to \mathbb{G}_m-\{1\}$ with respect to $[d]$,

Indeed, take integers
$ b = (b_1, \ldots, b_n)$ such that $\sum_{i=1}^nb_iw_i = 1$, and 
we take new variables $Y_i := \lambda^{b_i}X_i$.
Then
the equation $(\ref{eq:generalized})$ becomes 
 \begin{equation}
   \label{eq:descent}
   Q_t := w_1t^{-b_1}Y_1^d + \cdots + w_nt^{-b_n}Y_n^d - dY_1^{w_1}\cdots Y_n^{w_n} = 0.
\end{equation}
This gives rise to the descent family.

The group $\Gamma_W/\Delta$ also acts on the descent family, so that we 
have the eigensheaves $\Prim^{n-2}_{\text{\rm desc}}(V \bmod W)_{\overline\Q_\ell}$ of
the primitive part of 
$R^{n-2}\pi_{\text{\rm desc},*}\overline{\Q}_\ell$.
\begin{lemma}[{\cite[Lemma 6.2, Lemma 6.3.]{katz}}]
  Suppose that $\Prim^{n-2}_{\text{\rm desc}}(V \bmod W)_{\overline\Q_\ell}$ 
  is nonzero. Then there exists a continuous character
   \[\Lambda_{V, W}: \pi_1(\spec R_0[1/\ell]) \rightarrow \overline{\Q}_\ell^\times\]
   and the hypergeometric sheaf $\mathcal{H}_{V, W}$ such that
   \[\Prim^{n-2}_{\text{\rm desc}}(V \bmod W)_{\overline\Q_\ell} \cong \mathcal{H}_{V, W} \otimes \Lambda_{V, W}.\]
 \end{lemma}
 See \cite[8.17.6]{esde}, \cite[\S4, \S5, \S6]{katz} for 
 the definition of the hypergeometric sheaf $\mathcal{H}_{V, W}$. 
In this paper, we only use the following property.
Fix embeddings $R_0 \hookrightarrow \C \hookrightarrow \overline{\Q}_\ell$. Since $\mathcal{H}_{V, W}$ is constructed using the cancel operation \cite[9.3.1]{esde}, the analytification of
$\mathcal{H}_{V, W}$ on the complex analytic space 
$\mathbb{G}_{m, \C} - \{1\}$ is the local system 
corresponding to an irreducible hypergeometric $\mathcal{D}_{\mathbb{G}_{m, \C} - \{1\}}$-module \cite[8.17.11]{esde}.

\begin{lemma}\label{lemma:irreducibility}
   Let $\Prim^{n-2}_{\text{\rm desc}}(V \bmod W)_{\overline\Q}$ denote the eigencomponent
   of the primitive part of the local system $R^{n-2}\pi_{\text{\rm desc},*}\overline{\Q}$
   on 
   the complex analytic space 
$\mathbb{G}_{m, \C} - \{1\}$.
Then it is the local system corresponding to 
an irreducible hypergeometric $\mathcal{D}$-module.
In particular, $\Prim^{n-2}_{\text{\rm desc}}(V \bmod W)_{\overline\Q}$ is irreducible.
 \end{lemma}
 
\section{Picard-Fuchs equations of the generalized Dwork family}
\subsection{Jacobian rings of hypersurfaces}
Let $K$ be a field of characteristic zero.
Let $S$ be a smooth  affine variety over $K$ and let $f : X \rightarrow S$ be a smooth 
$K$-morphism.
Then
$R^if_*\Omega_{X/S}^\bullet$ is endowed a $\mathcal{D}_S$-module structure by the Gauss-Manin connection \cite{katzoda}
\[\nabla^\mathrm{GM} : R^if_*\Omega_{X/S}^\bullet \rightarrow R^if_*\Omega_{X/S}^\bullet \otimes \Omega_S^1.\]
For $\omega \in H^i_{\dR}(X/S)$, a differential operator $P \in \mathcal{D}_S$ which satisfies
\[ P\omega = 0 \]
is called a Picard-Fuchs equation of $\omega$.

\medskip  Let $Y$ be a nonsingular hypersurface of the complex projective space $\mathbb{P}^{n-1}_K$ given by a homogeneous polynomial $Q \in S:=K[X_1, \ldots, X_n]$ of degree $d$.
  The graded ring \[R_Q := S/J_Q,\quad J_Q:=(\frac{\partial Q}{\partial X_1}, \ldots, \frac{\partial Q}{\partial X_n})\] is called the Jacobian ring. We denote the homogeneous part of $R_Q$ of degree $N$ by $R^N _Q$.

\begin{theorem}[Griffiths, \cite{gri}]  
Put $\Omega = \sum_{i=1}^n(-1)^{i+1}X_idX_1 \wedge \cdots \wedge \widehat{dX_i} \wedge \cdots \wedge dX_n$.
Then
the map
  \begin{align*}
    S^{pd-n} &\longrightarrow H^{n-1}_\dR(\PP^{n-1} \backslash Y)\\
    A &\longmapsto \frac{A}{Q^p}\Omega
  \end{align*}
  induces an isomorphism
\[  
 R^{pd-n} _Q \overset{\cong}{\longrightarrow} 
 \mathrm{Gr}_F^{n-p}H^{n-1}_\dR(\PP^{n-1} \backslash Y)
 \]
 where $F^\bullet H^*_\dR(\PP^{n-1} \backslash Y)$ denotes the Hodge filtration.
  Furthermore, we have
  \begin{equation}
    \label{eq:griffiths}
    \frac{pA\frac{\partial Q}{\partial X_i}}{Q^{p+1}}\Omega =  \frac{\frac{\partial A}{\partial X_i}}{Q^{p}}\Omega
  \end{equation}
  as cohomology classes.
  \end{theorem}

  \begin{theorem}
    The residue map induces an isomorphism onto the primitive part 
\[ \mathrm{Res} : \mathrm{Gr}_F^{n-p}H^{n-1}(\PP^{n-1} \backslash Y) 
\overset{\cong}{\longrightarrow} H^{n-1-p, p-1}(Y)_{\mathrm{prim}}.\]
\end{theorem}

With the above theorems, we can calculate Picard-Fuchs equations. For details of the method, see \cite[\S 5.3]{CK}.

\subsection{Picard-Fuchs equations of the generalized Dwork family}\label{Picard-Fuchs.sect}
Let $(n, d, W)$ be the data as in \S \ref{defn-sect} and put $d_W = lcm(w_1, \ldots, w_n)d$. 
Let $K$ be a field of charcteristic zero containing primitive $d_W$-th roots of unity, and let $\pi:\mathbb{X} \rightarrow \mathbb{A}^1$
be the generalized Dwork family 
over $K$ defined by the equation \eqref{eq:generalized}.
By Lemma \ref{lemma:smooth}, $\pi$ is smooth on the affine variety
\begin{equation}
  \label{eq:DefinitionofU}
  U := \spec K[\lambda, (1-\lambda^d)^{-1}]\hookrightarrow \mathbb{A}^1_K.
\end{equation}
We often write $\pi^{-1}(U)\to U$ by $\mathbb{X}\to U$ if there is no fear of confusion.

Let $\Prim^{n-2}_{\dR}$ be the primitive part of $H^{n-2}_{\dR}(\mathbb{X}/U)$
defined in the same way as in \S \ref{defn-sect}, and
let 
\[
\Prim^{n-2}_{\dR}=\bigoplus_{V}\Prim^{n-2}_{\dR}(V\bmod W)
\]
be the eigendecomposition by $\Gamma_W/\Delta$.
The residue map 
\[\mathrm{Res}: H^{n-1}_{\dR}((\PP-\mathbb{X})/U) \xrightarrow{\sim} \Prim^{n-2}_{\dR}\]
is horizontal and $\Gamma_W/\Delta$-equivariant (e.g. \cite[\S 8]{katz}).
By this isomorphism we identify both sides.

\begin{theorem}
  \label{theorem:picard-fuchs} 
  Let $V=(v_1, \ldots, v_n) \in (\Z/d\Z)^n_0$ be a totally nonzero element.
  Let $D_\lambda = \lambda\frac{d}{d\lambda} = \lambda\nabla^{\mathrm{GM}}_{d/d \lambda}$.
  Put $p := \deg V$ and
  \begin{align}
    \omega_V &:= \mathrm{Res}\left(\frac{X_1^{\tilde{v}_1-1}\cdots X_n^{\tilde{v}_n-1}}{Q_\lambda^p}\Omega\right) \in \Prim^{n-2}_{\dR}(V\bmod W),
     \label{theorem:picard-fuchs.eq1} 
 \\
    P'(V, W) &:= \prod_{k=0}^{d-1}(D_\lambda-k)-\lambda^d\prod_{i=1}^n\prod_{j=0}^{w_i-1}\left(D_\lambda+\frac{\tilde{v}_i+jd}{w_i}\right)
     \label{theorem:picard-fuchs.eq2} 
  \end{align}
  where $\tilde{v}_i\in\{0,1,\ldots,d-1\}$ denotes the unique integer 
  such that $\tilde{v}_i\equiv v_i\bmod d$.
 Then $P'(V, W)\omega_V = 0$.
\end{theorem}
       \begin{proof}	
	Let $\mu=(\mu_1, \ldots, \mu_{n+1})$ be parameters and let
	\begin{align*}
          Q_{\mu} &= -d\mu_1X_1^d- \cdots - d\mu_nX_n^d -d\mu_{n+1} X_1^{w_1}\cdots X_n^{w_n},\\
          \omega_\mu &= \frac{X_1^{\tilde{v}_1-1}\cdots X_n^{\tilde{v}_n-1}}{Q_{\mu}^p}\Omega,\\
          D_i &= \mu_i\frac{\partial}{\partial \mu_i}.
        \end{align*}
         We show that $\omega_\mu$ satisfies the GKZ system
        given by
	\[n\times (n+1)\text{-matrix}\  (a_{ij})=\begin{pmatrix}
		d & \cdots & 0 & w_1 \\
		\vdots & \ddots & \vdots & \vdots \\
		0 & \cdots & d & w_n 
	\end{pmatrix}\ \text{and}\   
	\begin{pmatrix}
		\tilde{v_1} \\ \vdots \\ \tilde{v_n}
          \end{pmatrix}\]
i.e. $\omega_\mu$ satisfies the system
         \begin{empheq}[left={\empheqlbrace}]{align}
           &\left(\frac{\partial}{\partial \mu_{n+1}}\right)^d\omega_\mu=\prod_{i=1}^n\left(\frac{\partial}{\partial\mu_i}\right)^{w_i}\omega_\mu \label{eq:1st}\\
           &\left(\sum_{j=1}^{n+1}a_{ij}D_j+\tilde{v_i}\right)\omega_\mu = 0 \hspace{10pt} (i = 1,\ldots, n)\label{eq:2nd}.
         \end{empheq}

         We can easily check that
         \[\left(\frac{\partial}{\partial\mu_{n+1}}\right)^d\omega_\mu=p(p+1)\cdots (p+d-1)\frac{X_1^{\tilde{v}_1-1+dw_1}\cdots X_n^{\tilde{v}_n-1+dw_n}}{Q_{\mu}^{p+d}}\Omega =\prod^n_{i=1}\left(\frac{\partial}{\partial\mu_i}\right)^{w_i}\omega_\mu,\]
	thus we obtain the equation $(\ref{eq:1st})$.

	$(\ref{eq:2nd})$ can be written as
	\[\left(\begin{pmatrix}dD_1 \\ \vdots \\ dD_n\end{pmatrix}+\begin{pmatrix}w_1D_{n+1} \\ \vdots \\ w_nD_{n+1}
                                                              \end{pmatrix}+\begin{pmatrix}\tilde{v_1} \\ \vdots \\ \tilde{v_n}\end{pmatrix}\right)\omega_\mu=0.\]
	 For $1 \leq i \leq n$
	\begin{align*}
		dD_i\omega_\mu+w_iD_{n+1}\omega_\mu&=pd^2\mu_iX_i^d\frac{X_1^{\tilde{v}_1-1}\cdots X_n^{\tilde{v}_n-1}}{Q_{\mu}^{p+1}}\Omega+pw_id\mu_{n+1} X_1^{w_1}\cdots X_n^{w_n}\frac{X_1^{\tilde{v}_1-1}\cdots X_n^{\tilde{v}_n-1}}{Q_{\mu}^{p+1}}\Omega\\
		&=-p\frac{\partial Q_{\mu}}{\partial X_i}\frac{X_1^{\tilde{v}_1-1}\cdots X_i^{\tilde{v}_i}\cdots X_n^{\tilde{v}_n-1}}{Q_{\mu}^{p+1}}\Omega\\
		&\overset{(\ref{eq:griffiths})}{=}-\tilde{v}_i \omega_\mu,
	\end{align*}
	we get the equation $(\ref{eq:2nd})$.

	From $(\ref{eq:1st})$, note that $(\frac{\partial}{\partial\mu_i})^{m}=\left(\mu_i^{-1}D_i\right)^m=\mu_i^{-m}\prod_{j=0}^{m-1}(D_i-j)$, we have
		\[\mu_{n+1}^{-d}\prod_{k=0}^{d-1}(D_{n+1}-k)\omega_\mu=\prod_{i=1}^n\mu_i^{-w_i}\prod_{j=0}^{w_i-1}(D_i-j)\omega_\mu.
	\]
        Moreover  by the equation $(\ref{eq:2nd})$ the right hand side is equal to
	 \begin{align*}
	 	\prod_{i=1}^n\mu_i^{-w_i}\prod_{j=0}^{w_i-1}\left(-\frac{w_i}{d}D_{n+1}-\frac{\tilde{v}_i}{d}-j\right)\omega_\mu=
	 	\prod_{i=1}^n\left(-\frac{d}{w_i}\mu_i\right)^{-w_i}\prod_{j=0}^{w_i-1}\left(D_{n+1}+\frac{\tilde{v}_i+jd}{w_i}\right)\omega_\mu.
	 \end{align*}
         By putting $\mu = (-w_1/d, \ldots, -w_n/d, \lambda)$ we obtain the desired equation.
       \end{proof}
       
       Fix $b=(b_1, \ldots, b_n) \in \Z^{n}$ such that $\sum^n_{i=1}b_iw_i = 1$, and let $\mathbb{Y}/\mathbb{G}_m-\{1\}$ be the descent family defined by the equation (\ref{eq:descent}).

         \begin{definition}
    Let $m\geq 1$ be an integer and let  $\alpha_1, \ldots, \alpha_m, \beta_1, \ldots, \beta_m$ be rational numbers. The differential operator 
\begin{equation}
  \label{eq:hypergeometric}
  \prod_{i=1}^{m}(D_t + \beta_i -1) - t\prod_{i=1}^m(D_t + \alpha_i)
\end{equation}
on $\mathbb{G}_m - \{1\}$ is called a hypergeometric operator (of type $(m, m)$).
$\alpha_i, \beta_i$ are called parameters and we denote the hypergeometric operator of the parameters $\alpha_1, \ldots, \alpha_m, \beta_1, \ldots, \beta_m$ as
\[Hyp
  \left(
  \begin{matrix}
    \alpha_1, &\ldots, &\alpha_m\\ \beta_1, &\ldots, &\beta_m
  \end{matrix}
  ; t \right) \text{ or  }Hyp(\alpha_1, \ldots, \alpha_m; \beta_1, \ldots, \beta_m; t).\]
\end{definition}

    \begin{proposition}\label{proposition:descent}
Let $V \in (\Z/d\Z)_0^n/\langle W\rangle$ be a totally nonzero element.
Put $N = \sum_{i=1}^nb_i\tilde{v}_i$. Then \[\omega_{V, t} := \mathrm{Res}
  \left(\frac{Y_1^{\tilde{v}_1-1}\cdots Y_n^{\tilde{v}_n-1}}{Q_t^{\deg V}}\Omega\right) \in \Prim^{n-2}_{\dR, \text{\rm desc}}(V \bmod W)\] is annihilated by the hypergeometric operator
  \begin{align*}
  Hyp'(V, W, b) :&= Hyp
  \left(
  \begin{matrix}
\frac{N}{d}, &\frac{1}{d} + \frac{N}{d}, & \ldots, &\frac{d-1}{d} + \frac{N}{d}\\
\frac{w_1d - \tilde{v}_1}{w_1d} +\frac{N}{d},&\frac{(w_2-1)d - \tilde{v}_2}{w_2d} +\frac{N}{d}, & \ldots, &\frac{d - \tilde{v}_n}{w_nd} +\frac{N}{d}
  \end{matrix}
                  ; t \right)\\
                &=\prod_{i=1}^n\prod_{j=0}^{w_i-1}\left(D_t+\frac{(w_i- j)d - \tilde{v}_i}{w_id} +\frac{N}{d} - 1\right)-t\prod_{k=0}^{d-1}\left(D_t + \frac{k}{d} + \frac{N}{d}\right).
    \end{align*}
\end{proposition}
\begin{proof}
  Since the equation (\ref{eq:descent}) is obtained by putting $Y_i= \lambda^{b_i}X_i$ and $t = \lambda^{-d}$,  the section 
  \[\lambda^N\omega_V = \frac{(\lambda^{b_1}X_1)^{\tilde{v}_1-1}\cdots (\lambda^{b_n}X_n)^{\tilde{v}_n-1}}{Q_\lambda}\sum_{i=1}^n(-1)^{n+1}\lambda^{b_i}X_id(\lambda^{b_1}X_1)\cdots \widehat{d(\lambda^{b_i}X_i)}\cdots d(\lambda^{b_n}X_n)\]of $\Prim^{n-2}_{\dR}(V, W)|_{U-\{0\}}$
    descents to $\omega_{V, t}$.
    From Theorem \ref{theorem:picard-fuchs} we have
    \begin{align*}
      0 &= (P'(V, W)\lambda^{-N})\lambda^N\omega_V\\
    &= \lambda^{-N}\left(\prod_{k=0}^{d-1}(D_\lambda-k -N)-\lambda^d\prod_{i=1}^n\prod_{j=0}^{w_i-1}\left(D_\lambda+\frac{\tilde{v}_i+jd}{w_i}-N\right)\right)\lambda^N\omega_V .
    \end{align*}
  Since the action of $D_\lambda$ on $\Prim^{n-2}_{\dR}(V, W)|_{U-\{0\}}$ corresponds to $-d D_t$, by variable transforming and dividing by a power of $-d$, we can see that $\omega_{V, t}$ is annihilated by $Hyp'(V, W, b)$.
  \end{proof}
  \begin{definition}
    For a hypergeometric differential operator
\[  Hyp(\alpha_1, \ldots, \alpha_m; \beta_1, \ldots, \beta_m ; t) = \prod_{i=1}^{m}(D_t + \beta_i -1) - t\prod_{i=1}^m(D_t + \alpha_i)
\]
whose parameters $(\alpha_1, \cdots, \alpha_r, \cdots, \alpha_m; \beta_1, \cdots, \beta_r, \cdots,\beta_{m})$ are ordered so that $\alpha_i \equiv \beta_i \bmod \Z$ for all $i > r$, we define the cancel operation as
\[\mathrm{\bf Cancel}\ Hyp(\alpha_1, \cdots, \alpha_m; \beta_1, \cdots, \beta_{m}; t) := Hyp(\alpha_1, \cdots, \alpha_r; \beta_1, \cdots, \beta_{r}; t).\]
\end{definition}

\begin{lemma}\label{lemma:cancel}
Let $0 \leq j \leq w_i-1, 0 \leq k \leq d-1$ and $1\leq \tilde{v}_i\leq d-1$
be integers. 
Then, the condition
  \[\frac{(w_i- j)d - \tilde{v}_i}{w_id} \equiv \frac{k}{d} \mod \Z \]
    implies
   \[\frac{(w_i- j)d - \tilde{v}_i}{w_id} = \frac{k}{d}\]
In particular, the cancel operation removes pairs of identical parameters from $Hyp'(V, W, b)$
in Proposition \ref{proposition:descent}.
\end{lemma}

\begin{proof}
This is immediate as
  \[\frac{(w_i- j)d - \tilde{v}_i}{w_id} - \frac{k}{d}=1-\frac{jd+\tilde{v}_i+w_ik}{w_id}
  <1,\]
  \[\frac{(w_i- j)d - \tilde{v}_i}{w_id} - \frac{k}{d}
  =\frac{d - k}{d} - \frac{\tilde{v}_i+jd}{w_id}
   \geq\frac{1}{d} - \frac{w_id - 1}{w_id}>-1.\]
\end{proof}
  \begin{theorem}\label{theorem:picard-fuchs-descent}
    Under the setting of Proposition \ref{proposition:descent}, $\omega_{V, t}$ annihilated by
    \[Hyp(V, W, b) := \mathrm{\bf Cancel}\  Hyp'(V, W, b).\]
  \end{theorem}
  \begin{proof}
    Suppose that $\omega \in \Prim^{n-2}_{\dR, \text{\rm desc}}(V \bmod W)$ is annihilated by the operator $D_t - c$ for some $c \in \Q$. If $\omega \neq 0$, the submodule generated by $\omega$ is isomorphic to $\mathcal{D}/\mathcal{D}(D_t - c)$. 
    Since $\Prim^{n-2}_{\dR, \text{\rm desc}}(V \bmod W)$ is irreducible by
Lemma \ref{lemma:irreducibility}, 
we have $\Prim^{n-2}_{\dR, \text{\rm desc}}(V \bmod W)\cong \mathcal{D}/\mathcal{D}(D_t - c)$, and hence
$\mathcal{D}/\mathcal{D}(D_t - c)\cong \mathcal{D}/\mathcal{D}Hyp(\alpha; \beta; t)$ whose parameters satisfy the condition $\alpha - \beta \notin \Z$. 
     This is impossible. This shows $\omega = 0$.
    
              Write
          \[Hyp'(V, W, b) = Hyp(\alpha_1, \ldots, \alpha_d; \beta_1, \ldots, \beta_d; t)\]
          so that $\alpha_i = \beta_i$ for each $i > r$. 
          Since
                    \begin{align*}
                    &Hyp(\alpha_1, \ldots, \alpha_d; \beta_1, \ldots, \beta_d; t) \\
                    = 
                    &(D_t  + \beta_i -1)Hyp(\alpha_1, \ldots, \widehat{\alpha_i}, \ldots, \alpha_d; \beta_1, \ldots, \widehat{\beta_i}, \ldots,\beta_d; t),\end{align*}
the element
          \[Hyp(\alpha_1, \ldots, \widehat{\alpha_i}, \ldots, \alpha_d; \beta_1, \ldots, \widehat{\beta_i}, \ldots,\beta_d; t)\omega_{V, t} 
          \in \Prim^{n-2}_{\dR, \text{\rm desc}}(V \bmod W)\]
is annihilated by $(D_t  + \beta_i -1)$, and hence it vanishes.
          Repeating this argument, we have that $Hyp(V, W, b)\omega_{V, t} = 0$ by Lemma \ref{lemma:cancel}.
       \end{proof}
       \begin{example}
         In the case of $V = W$, $N = \sum_{i=1}^nb_i\tilde{v}_i =\sum_{i=1}^nb_iw_i = 1$. Then the section \[\omega_{W, t} = \mathrm{Res} \left(\frac{Y_1^{w_1-1}\cdots Y_n^{w_n-1}}{Q_t}\Omega\right)\] is annihilated by
         \[\mathrm{\bf Cancel }\ Hyp
  \left(
  \begin{matrix}
\frac{1}{d}, &\frac{2}{d}, & \ldots, &1\\ \frac{w_1}{w_1},&\frac{(w_1-1)}{w_1}, & \ldots, &\frac{1}{w_n}
  \end{matrix}
  ; t \right).\]
This result corresponds to Lemma 8.3 of \cite{katz}.
\end{example}
We turn to the family $\mathbb{X}/U$.
       Let $V,W$ be as before, and 
       $\tilde{v}_i\in\{0,1,\ldots,d-1\}$ denotes the unique integer 
  such that $\tilde{v}_i\equiv v_i\bmod d$.
  Let
  \begin{align*}
             I(V, W) &:= \Z \cap \bigcup_{i=1}^n\left\{d - \frac{\tilde{v}_i}{w_i}, d -\frac{\tilde{v}_i + d}{w_i},\ldots, d - \frac{\tilde{v}_i + (w_i-1)d}{w_i}\right\}.
           \end{align*}
           If $j$ is an integer such that $0 \leq j \leq w_i - 1$, then 
           
           \[\frac{\tilde{v}_i+dj}{w_i} < \frac{d+d(w_i-1)}{w_i} =d,\]
and hence $I(V, W)$ is a subset of $\{0, 1, \ldots, d-1\}$.
           Let $P'(V, W)$ be the differential operator \eqref{theorem:picard-fuchs.eq2}.
           If $k\in I(V, W)$, then
           $P'(V, W)$ is factorized by $(D_\lambda -k)$ from left. Let $P(V, W)$ denote
           the differential equation obtained by removing all such factors from $P'(V, W)$, namely \[P'(V, W) =\prod_{k \in I(V, W)}(D_\lambda -k) P(V, W).\]

           \begin{lemma}\label{lemma:I=J} 
           Assume that $V \in (\Z/dZ)^n_0$ is totally nonzero.
Put
\[J(V, W) := \{r \in\{ 0,1, \ldots, d-1\} \mid \tilde{v}_i + r w_i \equiv 0\bmod d \text{ for some $i$}\}.\]
Then $I(V, W)=J(V,W)$.
           \end{lemma}
           \begin{proof}
             Obviously $I(V, W)\subset J(V, W)$.
             For an element $r \in J(V, W)$, let $s$ be the integer such that $\tilde{v}_i +r w_i =sd$. This satisfies $1 \leq s \leq w_i $ as $sd = \tilde{v}_i +r w_i < d+dw_i  = d(1+w_i )$. 
Hence we have \[r = \frac{-\tilde{v}_i +sd}{w_i} = d - \frac{\tilde{v}_i +(w_i -s)d}{w_i} \in I(V, W).\]
This shows $I(V, W)\supset J(V,W)$.
           \end{proof}

           \begin{theorem}\label{theorem:picard-fuchs-reduced}
             Under the setting of Theorem \ref{theorem:picard-fuchs}, we have
         \[P(V, W)\omega_V =0.\] 
       \end{theorem}
       \begin{proof}
  We notice that 
  \[
\lambda^{N} P'(V,W)\lambda^{-N}=\text{(constant)}\times  Hyp'(V, W, b),
  \]and
    the cancel operation for $Hyp'(V, W, b)$ corresponds to the cancel
  operation for $P'(V, W)$ by Lemma \ref{lemma:I=J}. 
    Therefore
    the assertion follows from Theorem \ref{theorem:picard-fuchs-descent}
    together with the fact that $\omega_{V,t}=\lambda^N\omega_V$.
       \end{proof}

\subsection{$\mathcal{D}$-module structure of $\Prim^{n-2}_{\dR}$.}

Let $\mathcal D$ be the ring of differential operators on $K[t,(t-t^2)^{-1}]$.
\begin{lemma}\label{lemma:RankIrr}
  Assume that $V$ is totally nonzero.
  \begin{enumerate}
  \item[\rm(1)] The order of $Hyp(V, W, b)$ coincides with the rank of 
   $\Prim^{n-2}_{\dR, \text{\rm desc}}(V\bmod W)$.
   \item[\rm(2)]  The $\mathcal{D}$-module $\mathcal{D}/\mathcal{D}Hyp(V, W, b)$ is irreducible.
  \end{enumerate}
\end{lemma}
\begin{proof}
    \begin{enumerate}[(1)]
  \item 
    By Lemma \ref{lemma:rank} and Lemma \ref{lemma:I=J}, we have
    \begin{align*}
      \mathrm{rank}\  \Prim^{n-2}_{\dR, \text{\rm desc}}(V\bmod W) &=\mathrm{rank}\  \Prim^{n-2}_{\dR}(V\bmod W)\\
                                                     &=  d-\#J(V, W)\\ &= d-\#I(V, W) \\ &= \text{The order of } P(V, W) \\
      &= \text{The order of } Hyp(V, W, b)
    \end{align*}
    as required.
  \item
    In general, the hypergeometric $\mathcal{D}$-module $\mathcal{D}/\mathcal{D}Hyp(\alpha_1, \ldots, \alpha_n; \beta_1, \ldots, \beta_n; t)$ is irreducible if and only if $\alpha_i \not\equiv \beta_j \bmod \Z $ for every $i, j$ \cite[Corollary 3.2.1]{esde}.
Hence $\mathcal{D}/\mathcal{D}Hyp(V, W, b)$ is irreducible by the definition of the cancel operation.
\end{enumerate}
\end{proof}

The following theorem is a generalization of Theorem 8.4 and Corollary 8.5 of \cite{katz}.
\begin{theorem}\label{theorem:iso}
  Assume that $V \in (\Z/d\Z)^n_0$ is totally nonzero and let
  \[\mathcal{H}_{V, W, b} := \mathcal{D}_{\mathbb{G}_m-\{1\}}/\mathcal{D}_{\mathbb{G}_m-\{1\}}Hyp(V, W, b).\]
  Then there exist isomorphisms of $\mathcal{D}$-modules
  \begin{align*}
  \Prim^{n-2}_{\dR, \text{\rm desc}}(V \bmod W) &\cong \mathcal{H}_{V, W, b},\\
  \Prim^{n-2}_{\dR}(V \bmod W)|_{U-\{0\}} &\cong [d]^*\mathcal{H}_{V, W, b},\\
  \Prim^{n-2}_{\dR}(V \bmod W) &\cong  j_{!*}[d]^*\mathcal{H}_{V, W, b}.
  \end{align*}
  where $j_{!*}$ is the minimal extension by the inclusion $j:U-\{0\} \rightarrow U$. 
\end{theorem}
\begin{proof}
It is enough to show the first isomorphism, because
the 2nd and 3rd isomorphisms can be derived from it. Consider a morphism
    \begin{align*}
    \mathcal{H}_{V, W, b} &\longrightarrow \Prim^{n-2}_{\dR, \text{\rm desc}}(V\bmod W).\\
    P &\longmapsto P\omega_{V,t}
    \end{align*}
    Since $\mathcal{H}_{V, W, b}$ is irreducible, this morphism is injective.
Moreover we know that both sides are locally free $\mathcal{O}_{\mathbb{G}_m-\{1\}}$-modules with the same rank by Lemma $\ref{lemma:RankIrr}$, hence the morphism is surjective.
\end{proof}

\section{Application to $p$-adic cohomology theory}

\subsection{Differential modules and Matrix calculus}
We mean by a differential ring a commutative ring $R$ equipped with an additive map $d: R \rightarrow R$ satisfying the Leibniz rule
  \[d(ab) = ad(b) + bd(a),\quad (a,b\in R).\]
The map $d$ is called a derivation.
For a matrix $A = (a_{ij})$ with $a_{ij}\in R$, we write
$d(A) = (d(a_{ij}))$.
  A differential module over a differential ring $(R, d)$ is a $R$-module $M$ equipped with an additive map $D : M \rightarrow M$ satisfying
  \[D(am) = aD(m) + d(a)m,\quad(a\in R,\, m\in M).\]
The map $D$ is called a differential operator on $M$.

\medskip

  Let $(M, D)$ be a finite free differential module of rank $n $ over a differential ring $(R, d)$. Let $e_1, \ldots, e_n$ be a basis of $M$. Define the matrix of action of $D$ 
  with respect to the basis $e_1, \ldots, e_n$ to be the $n \times n$ matrix $C = (c_{ij})$ with
  $c_{ij}\in R$
  given by 
  \[D(e_j) = \sum^n_{i=1}c_{ij}e_i.\]

\begin{proposition}\label{matrixlemma}
  Let $(M, D)$ be a finite free differential module of rank $n $ over a differential ring $(R, d)$.
  \begin{enumerate}[(1)]
  \item[\rm(1)]   For an invertible matrix $A = (a_{ij})$ with $a_{ij}\in R$, we have
  \begin{equation}
    \label{eq:derivationofinvertiblematrix}
    d(A^{-1}) = -A^{-1} d(A)A^{-1}.
  \end{equation}
\item[\rm(2)] Let $C, \widetilde{C}$ be the matrices of the action of $D$ with respect to
bases $e_1, \ldots, e_n$ and $\tilde{e}_1, \ldots, \tilde{e}_n$ of $M$ respectively. 
Then
    \begin{equation}
      \label{eq:basechangeofconnection}
      \tilde{C} = B^{-1}CB + B^{-1}d(B),
\end{equation}
  where $B$ is 
  the change-of-basis matrix from $e_1, \ldots, e_n$ to $\tilde{e}_1, \ldots, \tilde{e}_n$.
  \item[\rm(3)]
  Let $C,\widetilde C$ and $B$ be as above.
  Then
  \[
  d(X)-XC=O\iff d(BX)-(BX)\widetilde C=O.
  \]
  \end{enumerate}
\end{proposition}
\begin{proof}
Straightforward.
\end{proof}

\subsection{Deformation matrix of Dwork families}
We compute the Frobenius matrix on the Dwork families
by the deformation method.
We refer the reader to \cite{gerkmann}, \cite{kedlaya} 
for general references of deformation method.

\bigskip

Let $p>0$ be a prime number. 
 Let $W=W(\overline\F_p)$ be the ring of Witt vectors
 of the algebraic closure $\overline\F_p$, and $K$ the fractional field.
Let $W[X_1, \ldots, X_m]^\dagger$ be the ring of power series
$\sum a_\alpha X^\alpha \in W[[X_1, \ldots, X_m]]$ 
such that $|a_\alpha|r^{|\alpha|} \rightarrow 0$
as  $|\alpha| \rightarrow \infty$ for some $r >1$.
  For a $\Z_q$-algebra $B$ of finite type, 
  the weak completion $B^\dag$ is defined to be $
    W[X_1, \ldots, X_m]^\dagger/I  W[X_1, \ldots, X_m]^\dagger$ for a
presentation $B \cong W[X_1, \ldots, X_m]/I$. We write $B^\dag_K:=
B^\dag\otimes\Q$.
 
 \medskip
 
Fix integers $d \geq n \geq 3$ and $w_1, \ldots, w_n\geq 1$ such that the condition (\ref{eq:conditions}) is satisfied and assume that
$d, w_1, \ldots, w_n$ are prime to $p$. Consider the generalized Dwork family $\mathbb{X} \rightarrow U =\spec W[\lambda, (1-\lambda^d)^{-1}] $
defined by the equation \eqref{eq:generalized}, and 
the family of complements $(\PP - \mathbb{X})/U$.
Put $B = \Gamma(U, \mathcal{O}_U)$. Let
$\sigma: B^\dagger \rightarrow B^\dagger$ be the 
continuous $F_W$-linear ring homomorphism such that $\lambda\mapsto\lambda^p$
where $F_W$ is the $p$-th Frobenius on $W$.
Let $H_{\rig}^*((\PP-\mathbb{X})_{\overline\F_p}/U_{\overline\F_p})$ be the rigid cohomology group.
This is a $B^\dagger_{K}$-module, and it is
endowed the following structure,
\begin{itemize}
\item the integrable connection 
\[\nabla^{\rig}:H_{\rig}^i((\PP-\mathbb{X})_{\overline\F_p}/U_{\overline\F_p}) \rightarrow H_{\rig}^i((\PP-\mathbb{X})_{\overline\F_p}/U_{\overline\F_p}) \otimes \Omega^1_{B^\dagger_{K}}\]
induced from the Gauss-Manin connection
under the comparison
\begin{equation}\label{rig.dR}
H_{\rig}^i((\PP-\mathbb{X})_{\overline\F_p}/U_{\overline\F_p})\cong 
B^\dag_{K}\otimes_BH^i_\dR((\PP-\mathbb{X})_{K}/U_{K})
\end{equation}
with the de Rham cohomology group,
\item an isomorphism  
\begin{equation}\label{isogeny}
\Phi: \sigma^*H_{\rig}^i((\PP-\mathbb{X})_{\overline\F_p}/U_{\overline\F_p}) \xrightarrow{\sim} H_{\rig}^i((\PP-\mathbb{X})_{\overline\F_p}/U_{\overline\F_p})
\end{equation}
of $B^\dag_{K}$-modules which commutes with $\nabla^{\rig}$.
\end{itemize}

Let $\omega_1, \ldots, \omega_r$ be a basis of $H^{n-1}_{\rig}((\PP-\mathbb{X})_{\overline\F_p}/U_{\overline\F_p})$ . Let $C(\lambda) = (c_{ij}(\lambda))$ be the matrix of the action of $\frac{d}{d\lambda} := \nabla^{\rig}_{d/d\lambda}$,
and let $F(\lambda)=(f_{ij}(\lambda))$ be the matrix of the action of $\Phi$,
\[
\frac{d}{d\lambda} \omega_j = \sum^r_{i=1}c_{ij}(\lambda)\omega_i, \quad
\Phi\omega_j = \sum^r_{i=1}f_{ij}(\lambda)\omega_i, 
\] 
The commutativity of $\nabla^{\rig}$ and $\Phi$ means that $F(\lambda)$
satisfies a differential equation
\begin{equation}
  \label{eq:commutativity-differentialeq}
   \frac{d}{d\lambda}F(\lambda) +C(\lambda)F(\lambda) = 
   p\lambda^{p-1}F(\lambda)C(\lambda^p).
 \end{equation}
Let $A(\lambda)$ be the unique solution of the differential equation
\begin{equation}
  \label{eq:deformation_matrix}
  \frac{d}{d\lambda}A(\lambda) =  A(\lambda)C(\lambda)
\end{equation}
with the initial condition
\begin{equation}
  \label{eq:initialcondition}
  A(0) = I.
\end{equation}
The matrix $A(\lambda)$ is called the deformation matrix 
with respect to the basis $\omega_1,\ldots,\omega_r$.
Straightforward calculation shows that
$X(\lambda) := A(\lambda)F(\lambda)A(\lambda^p)^{-1}$ satisfies
$\frac{d}{d\lambda}X(\lambda)=O$, so that $X(\lambda)$ is a constant matrix, namely
$X=F(0)$,
\begin{equation}
  \label{eq:deformation_matrix2}
F(\lambda) = A(\lambda)^{-1}F(0)A(\lambda^p).
\end{equation}

\medskip

Recall the decomposition
\[
H^{n-1}_\dR((\PP-\mathbb{X})_{K}/U_{K})\cong
\Prim^{n-2}_{\dR} = \bigoplus_{V} 
\Prim^{n-2}_{\dR}(V\bmod W)
\]
from \S \ref{Picard-Fuchs.sect}.
Suppose that
$\Prim^{n-2}_{\dR}(V\bmod W)$ 
has a basis of 
cyclic vectors $\omega_V, \frac{d}{d\lambda}\omega_V, \ldots, \frac{d^{r-1}}{d\lambda^{r-1}}\omega_V$.
Let $C(\lambda)_V$ denote the matrix of the action of $\frac{d}{d\lambda}$
on $\Prim^{n-2}(V\bmod W)_{\dR}$, which
is of the form
\[C(\lambda)_V =
  \left(\begin{matrix}
          0 &  \cdots & 0 & -c_{0}(\lambda)\\
          1 &\cdots & 0 & -c_{1}(\lambda)\\
          \vdots & \ddots & \vdots & \vdots\\
          0& \cdots & 1 & -c_{r-1}(\lambda)
  \end{matrix}\right)
\]
where $c_1(\lambda), \ldots, c_r(\lambda)$ are the coefficients of the Picard-Fuchs equation
\begin{equation}
\label{eq:PF.matrix2}
\frac{d^r}{d\lambda^r}\omega_V+c_{r-1}(\lambda)\frac{d^{r-1}}{d\lambda^{r-1}}\omega_V 
+ \cdots +c_0(\lambda)\omega_V = 0.
\end{equation}
Moreover the deformation matrix $A(\lambda)_V$
with respect to 
the basis $\{\frac{d^i}{d\lambda^i}\omega_V\}_i$ is given as follows,
\begin{equation}
\label{eq:deformation_matrixV}
  \frac{d}{d\lambda}A(\lambda)_V =  A(\lambda)_VC(\lambda)_V.
\end{equation}
Let $\{w_0(\lambda)_V, \ldots, w_{r-1}(\lambda)_V\}$  be a fundamental set of solutions of 
the Picard-Fuchs equation \eqref{eq:PF.matrix2}
around $\lambda = 0$. Then the Wronskian matrix 
\begin{equation}\label{Wronskian matrix}
W(\lambda)_V = 
    \begin{pmatrix}
      w_0(\lambda)_V & \frac{d}{d\lambda}w_0(\lambda)_V & \cdots & \frac{d^{r-1}}{d\lambda^{r-1}}w_0(\lambda)_V\\
      \vdots &\vdots& \cdots &\vdots\\
      w_{r-1}(\lambda)_V & \frac{d}{d\lambda}w_{r-1}(\lambda)_V & \cdots & \frac{d^{r-1}}{d\lambda^{r-1}}w_{r-1}(\lambda)_V
    \end{pmatrix}
  \end{equation}
is a solution of the differential equation \eqref{eq:deformation_matrixV}, so that 
one has
\begin{equation}
  \label{eq:deformation_matrix4}
A(\lambda)_V=W(0)_V^{-1}W(\lambda)_V
\end{equation}
(note that the matrix $W(0)_V$ is invertible as $w_0(\lambda)_V,\ldots,w_{r-1}(\lambda)_V$ are
linearly independent over $K$). 

\medskip

Put $V^{(1)}=(p^{-1}v_1, \ldots, p^{-1}v_n) \in (\Z/d\Z)^n_0$.
Since the Frobenius $\Phi$ commutes with the action of $\Gamma_W/\Delta$,
we have 
\[
\Phi(\Prim^{n-2}_{\dR}(V^{(1)}\bmod W))
\subset B^\dag_{K}\otimes \Prim^{n-2}_{\dR}(V\bmod W).
\]
Thanks to the isomorphism \eqref{isogeny},
$\Prim^{n-2}_{\dR}(V^{(1)}\bmod W)$ 
also has a basis of cyclic vectors $\{\frac{d^i}{d\lambda^i}\omega_{V^{(1)}}\}_i$.
Let $F(\lambda)_V$ be the matrix of $\Phi$ with respect to 
the bases $\{\frac{d^i}{d\lambda^i}\omega_{V^{(1)}}\}_i$
and $\{\frac{d^j}{d\lambda^j}\omega_{V}\}_j$. 
It follows from \eqref{eq:deformation_matrix2} that we have
\begin{equation}
  \label{eq:deformation_matrix3}
F(\lambda)_V = A(\lambda)^{-1}_VF(0)_VA(\lambda^p)_{V^{(1)}}.
\end{equation}

\medskip

Suppose that one can take $\omega_V$ 
to be the differential form \eqref{theorem:picard-fuchs.eq1} in the above setting.
Then the Picard-Fuchs equation is given by the hypergeometric 
equation $P(V,W)$ 
by Theorem \ref{theorem:picard-fuchs-reduced}.
Write
\[
    P(V, W)= \prod_{i=0}^{r-1}(D_\lambda-k_i)-\lambda^d
    \prod_{j=0}^{r-1}(D_\lambda+\alpha_j),
\]
where we note that $k_0,\ldots,k_r$ are pairwise distinct integers such that $0\leq k_i<d$
and $\alpha_j\in\{(\tilde v_k+ld)/w_k\mid 1\leq k\leq n,\, 0\leq l<w_k\}$. 
We can obtain a fundamental set $\{w_i(\lambda)_V\}_i$
of solutions by the hypergeometric series (cf. \cite[16.8.6]{nist}),
namely we set
\begin{equation}\label{eq:w-series}
w_i(\lambda)_V:=\lambda^{k_i}{}_rF_{r-1}\left({\frac{\alpha_0+k_i}{d},\ldots,\frac{\alpha_{r-1}+k_i}{d}
\atop 1+\frac{-k_0+k_i}{d},\ldots,\widehat{1+\frac{-k_i+k_i}{d}},\ldots,1+\frac{-k_{r-1}+k_i}{d}}
;\lambda^d\right)
\end{equation}
for $i=0,1,\ldots,r-1$, and let $W(\lambda)_V$ be defined by \eqref{Wronskian matrix}.
Then the deformation matrix $A(\lambda)_V$ is given by \eqref{eq:deformation_matrix4}.

\medskip

Summing up the above, we have the following theorem.
\begin{theorem}\label{4.2.theorem}
Let $\omega_V$
be the differential form \eqref{theorem:picard-fuchs.eq1}.
Let $\{w_i(\lambda)_V\}_i$ be the hypergeometric series \eqref{eq:w-series}.
Suppose that
$\Prim^{n-2}_{\dR}(V\bmod W)$ 
has a basis  
$\{\frac{d^i}{d\lambda^i}\omega_V\}_i$.
Then the deformation matrix $A(\lambda)_V$ 
satisfies
\[
A(\lambda)_V=W(0)_V^{-1}W(\lambda)_V
\]
where $W(\lambda)_V$ is the Wronskian matrix \eqref{Wronskian matrix}.
Let $\{\omega_{V^{(1)}}\}_i$ and $\{w_i(\lambda)_{V^{(1)}}\}_i$ be defined in the same way
from $V^{(1)}$.
Then the same thing holds for 
$A(\lambda)_{V^{(1)}}$.
Let $F(\lambda)_V$ be the matrix of the Frobenius
\[
\Phi:
B^\dag_{K}\otimes\Prim^{n-2}_{\dR}(V^{(1)}\bmod W)\longrightarrow
 B^\dag_{K}\otimes \Prim^{n-2}_{\dR}(V\bmod W).
\]
with respect to 
the bases $\{\frac{d^i}{d\lambda^i}\omega_{V^{(1)}}\}_i$
and $\{\frac{d^j}{d\lambda^j}\omega_{V}\}_j$. 
Then it satisfies
\[
F(\lambda)_V=A(\lambda)^{-1}_VF(0)_VA(\lambda^p)_{V^{(1)}}.\]
\end{theorem}
The matrix $F(0)_V$ is the Frobenius on the Fermat variety, so 
it is more or less computable.

\medskip

The computation of the deformation matrices for the Dwork families 
was first given by Kloosterman \cite{kloosterman}.
He uses
another basis $\{\omega_{V_1}, \ldots, \omega_{V_r}\}$
of $\Prim^{n-2}_{\dR}$, where $V_1, \ldots, V_r$ are 
all elements of $(\Z/d\Z)^n_0$ which are totally nonzero, and 
he computes the deformation matrices
without using the Picard-Fuchs equations.

\begin{remark}
In Thereom \ref{4.2.theorem},
the condition 
\[
\text{``$\Prim^{n-2}_{\dR}(V\bmod W)$ is spanned by 
$\left\{\frac{d^i}{d\lambda^i}\omega_V\right\}_i$'' }
\]does {\it not} hold in general.
It is true that
$\Prim^{n-2}_{\dR}(V\bmod W)\otimes_{K[\lambda]}
K[\lambda^{-1}]$ is spanned by $\{\frac{d^i}{d\lambda^i}\omega_V\}_i$
by Theorem \ref{theorem:iso}, while one cannot
remove ``$\otimes
K[\lambda^{-1}]$'' in general.
If the condition fails, we need an additional argument to compute $A(\lambda)$.
This is illustrated in \S 4.3.
\end{remark}

\subsection{Example (\cite[\S 6 (j)]{dwork-padiccycles}, \cite[Example 5.7]{kloosterman})}
We compute the deformation matrix of the Dwork pencil of quartic K3 surfaces.
Let $p$ be an odd prime and let $\mathbb{X}/U$ be the Dwork family given by the data $(n, d, W) = (4, 4, (1, 1, 1, 1))$  over $W(\overline{\F}_p)$, i.e. the family defined by the equation
\[X_1^4 + X_2^4+ X_3^4+ X_4^4 -  4\lambda X_1X_2X_3X_4= 0.\]
In this case, $\{(1,1,1,1), (1,2,2,3), (1, 1, 3, 3)\}$ is a complete system of representatives of $(\Z/4\Z)^4_0/\langle W \rangle$.

Let $V_1 = (1, 2, 2, 3) \in (\Z/4\Z)^4_0$. By Theorem \ref{theorem:picard-fuchs} the Picard-Fuchs equation of $\omega_{V_1}$ is
\begin{align*}
  P(V_1, W) &= D_\lambda\cancel{(D_\lambda - 1)}\cancel{(D_\lambda - 2)}\cancel{(D_\lambda - 3)} - \lambda^4\cancel{(D_\lambda + 1)}(D_\lambda + 
2)\cancel{(D_\lambda + 2)}\cancel{(D_\lambda + 3)}\\
          &=D_\lambda - \lambda^4(D_\lambda + 2).
\end{align*}
Therefore ${}_1F_0\left( \begin{array}{c} \frac{1}{2}\\ -\end{array} ; \lambda^4
\right)$ is the solution of this equation with the initial condition ${}_1F_0\left( \begin{array}{c} \frac{1}{2}\\ -\end{array} ; 0
\right) =1$. Thus $A(\lambda)_{V_1} = {}_1F_0\left( \begin{array}{c} \frac{1}{2}\\ -\end{array} ; \lambda^4
\right)$ by Theorem \ref{4.2.theorem}.

\bigskip
Let $V_2 = (1, 1, 1, 1) \in (\Z/4\Z)^4_0$. The Picard-Fuchs equation of $\omega_{V_2}$ is
\begin{align*}
  P(V_2, W) &= D_\lambda(D_\lambda - 1)(D_\lambda - 2)\cancel{(D_\lambda - 3)} - \lambda^4\cancel{(D_\lambda + 1)}(D_\lambda + 1)(D_\lambda + 1)(D_\lambda + 1)\\
          &=D_\lambda(D_\lambda - 1)(D_\lambda - 2) - \lambda^4(D_\lambda + 1)^3.
\end{align*}
For this hypergeometric differential equation],
\begin{align*}
w_0(\lambda)_{V_2} = {}_3F_2\left( \begin{array}{c} \frac{1}{4} \; \frac{1}{4} \; \frac{1}{4} \\ \frac{1}{2} \; \frac{3}{4}\end{array} ; \lambda^4
\right), w_1(\lambda)_{V_2} = \lambda {}_3F_2\left( \begin{array}{c} \frac{1}{2} \; \frac{1}{2} \; \frac{1}{2} \\ \frac{3}{4} \; \frac{5}{4}\end{array} ; \lambda^4\right), w_2(\lambda)_{V_2} = 8\lambda^2 {}_3F_2\left( \begin{array}{c} \frac{3}{4} \; \frac{3}{4} \; \frac{3}{4} \\ \frac{5}{4} \; \frac{3}{2}\end{array} ; \lambda^4\right)
\end{align*}
form a fundamental set of solutions.
Then
\[W(\lambda)_{V_2} = 
    \begin{pmatrix}
      {}_3F_2\left( { \frac{1}{4} \; \frac{1}{4} \; \frac{1}{4} \atop \frac{1}{2} \; \frac{3}{4}} ; \lambda^4
      \right) & \frac{\lambda^3}{6} {}_3F_2\left( { \frac{5}{4} \; \frac{5}{4} \; \frac{5}{4} \atop \frac{3}{2} \; \frac{7}{4}} ; \lambda^4
                \right) & \frac{\lambda^2}{2} {}_3F_2\left( { \frac{5}{4} \; \frac{5}{4} \; \frac{5}{4} \atop \frac{3}{2} \; \frac{3}{4}} ; \lambda^4
                          \right)\\
      \lambda {}_3F_2\left( { \frac{1}{2} \; \frac{1}{2} \; \frac{1}{2} \atop \frac{3}{4} \; \frac{5}{4}} ; \lambda^4\right)& {}_3F_2\left( { \frac{1}{2} \; \frac{1}{2} \; \frac{1}{2} \atop \frac{3}{4} \; \frac{1}{4}} ; \lambda^4
                                                                                                                                                    \right) &\frac{8}{3}\lambda {}_3F_2\left( { \frac{3}{2} \; \frac{3}{2} \; \frac{3}{2} \atop \frac{7}{4} \; \frac{5}{4}} ; \lambda^4
                                                                                                                                                              \right)\\
      8\lambda^2 {}_3F_2\left( { \frac{3}{4} \; \frac{3}{4} \; \frac{3}{4} \atop \frac{5}{4} \; \frac{3}{2}} ; \lambda^4\right) & \lambda {}_3F_2\left( { \frac{3}{4} \; \frac{3}{4} \; \frac{3}{4} \atop \frac{5}{4} \; \frac{1}{2}} ; \lambda^4
                                                                                                                                                        \right) & {}_3F_2\left( { \frac{3}{4} \; \frac{3}{4} \; \frac{3}{4} \atop \frac{1}{4} \; \frac{1}{2}} ; \lambda^4
                                                                                                                                                                  \right)
    \end{pmatrix}
\]
satisfies the initial condition $W(0)_{V_2} = I$ and hence $A(\lambda)_{V_2} = W(\lambda)_{V_2}$
by Theorem \ref{4.2.theorem}.

\bigskip
For $V_3 = (1, 1, 3, 3)$,  the Picard-Fuchs equation of $\omega_{V_3}$ is
\begin{align*}
  P(V_3, W) &= D_\lambda\cancel{(D_\lambda - 1)}(D_\lambda - 2)\cancel{(D_\lambda - 3)} - \lambda^4\cancel{(D_\lambda + 1)}(D_\lambda + 1)\cancel{(D_\lambda + 3)}(D_\lambda + 3)\\
            &=D_\lambda(D_\lambda - 2) - \lambda^4(D_\lambda + 1)(D_\lambda + 3)\\
              &= \lambda^2(1-\lambda^4)\frac{d^2}{d\lambda^2} - (1+5\lambda)\frac{d}{d\lambda} - 3\lambda^4\frac{d}{d\lambda}.
\end{align*}
The fundamental solutions of this equation are
\[w_0(\lambda)_{V_3} = {}_2F_1\left( \begin{array}{c} \frac{1}{4} \; \frac{3}{4} \\ \frac{1}{2} \end{array} ; \lambda^4
\right), w_1(\lambda)_{V_3} = \lambda^2{}_2F_1\left( \begin{array}{c} \frac{3}{4} \; \frac{5}{4} \\ \frac{3}{2} \end{array} ; \lambda^4
\right)\]
and the Wronskian is
\[W(\lambda)_{V_3} = 
    \begin{pmatrix}
      {}_2F_1\left( {\frac{1}{4} \; \frac{3}{4} \atop \frac{1}{2} } ; \lambda^4
\right) & \frac{3}{2}\lambda^3{}_2F_1\left( {\frac{5}{4} \; \frac{7}{4} \atop \frac{3}{2} } ; \lambda^4
\right)\\\lambda^2{}_2F_1\left( {\frac{3}{4} \; \frac{5}{4} \atop \frac{3}{2} } ; \lambda^4
\right)&2\lambda\,{}_2F_1\left( {\frac{1}{4} \; \frac{3}{4} \atop \frac{1}{2} } ; \lambda^4
\right)
    \end{pmatrix}
  .\]

However, this is the case in which $\omega_{V_3}$ and $\frac{d}{d\lambda}\omega_{V_3}$ do not form a basis, so we cannot apply Theorem \ref{4.2.theorem} directly.

The eigenspace
$\Prim^{n-2}_{\dR}(V_3 \bmod W)$ is characterized as Deligne's canonical extension 
of $\Prim^{n-2}_{\dR}(V_3 \bmod W)|_{U_K - \{0\}}$ (\cite[II. \S5]{deligne}).
In this case, 
it is the coherent $\mathcal{O}_{U_K}$-submodule of $\Prim^{n-2}_{\dR}(V_3 \bmod W)|_{U_K - \{0\}}$ 
stable under $D_\lambda$ such that  
$(D_\lambda\bmod\lambda)$ is nilpotent.
We can compute it from the Picard-Fuchs equation $P(V_3, W)$.
Put $e_1 = \omega_{V_3}$ and $e_2 =\frac{1}{2\lambda}\frac{d}{d\lambda}\omega_{V_3}$,
then we have that
\begin{align*}
  D_\lambda (e_1) &=  2\lambda^2\, e_2\\
D_\lambda(e_2) &= \frac{3\lambda^2}{2(1-\lambda^4)}e_1 + \frac{6\lambda^4}{1-\lambda^4}e_2,
\end{align*}
and hence
\[
\mathcal{O}_{U_K}e_1\oplus \mathcal{O}_{U_K}e_2= \Prim^{n-2}_{\dR}(V_3 \bmod W).
\]
Since the change-of-basis matrix from $\{\omega_{V_3}, \frac{d}{d\lambda}\omega_{V_3}\}$ to $\{\omega_{V_3}, \frac{1}{2\lambda}\frac{d}{d\lambda}\omega_{V_3}\}$ is
\[B(\lambda) =
  \begin{pmatrix}
    1 & 0\\0& \frac{1}{2\lambda}
  \end{pmatrix}
\]
on $U_K - \{0\}$, the matrix
\begin{equation}
  \label{eq:result}
  W(\lambda)_{V_3}\,B(\lambda) = 
    \begin{pmatrix}
      {}_2F_1\left( {\frac{1}{4} \; \frac{3}{4} \atop \frac{1}{2} } ; \lambda^4
\right) & \frac{3}{4}\lambda^2{}_2F_1\left({\frac{5}{4} \; \frac{7}{4} \atop 
\frac{3}{2} } ; \lambda^4\right)\\
\lambda^2{}_2F_1\left( { \frac{3}{4} \; \frac{5}{4} \atop \frac{3}{2} } ; \lambda^4
\right)&{}_2F_1\left({\frac{1}{4} \; \frac{3}{4} \atop \frac{1}{2} } ; \lambda^4
\right)
    \end{pmatrix}
\end{equation}
is the solution of (\ref{eq:deformation_matrix}) for the basis $\{\omega_{V_3}, \frac{1}{2\lambda}\frac{d}{d\lambda}\omega_{V_3}\}$ by Proposition \ref{matrixlemma}.
Obviously $W(\lambda)_{V_3}\,B(\lambda)|_{\lambda = 0} = I$, that concludes
$A(\lambda)_{V_3} = W(\lambda)_{V_3}\,B(\lambda)$.


\end{document}